\documentclass[12pt]{amsart}
\usepackage{amsmath}
\usepackage{dsfont}\usepackage{color}
\usepackage{graphicx}\usepackage{amsbsy}
\usepackage{verbatim}	
\usepackage{amsthm}
\usepackage[margin=2.5cm]{geometry}
\usepackage{amssymb,amsmath}              
\usepackage{graphicx}
\usepackage{amssymb}
\usepackage{color}
\usepackage{epstopdf}
\renewcommand{\epsilon}{\varepsilon}
\DeclareMathOperator{\dvg}{div} \DeclareMathOperator{\spt}{spt}
\DeclareMathOperator{\dist}{dist} \DeclareMathOperator{\lip}{Lip}
 
\DeclareMathOperator{\trace}{trace}\DeclareMathOperator{\diam}{diam}
  
  \DeclareMathOperator{\proj}{proj}

     \DeclareMathOperator{\image}{Image}
    
\def\R{\mathbb{R}}
\def\C{\mathcal{C}}
\def\N{\mathbb{N}}

\def\d{\delta}
\def\a{\alpha}
\def\e{\epsilon}

\def\r{\rho}

\def\m{\mu}
\def\s{\sigma}
\def\o{\omega}
\def\l{\lambda}
\def\k{\kappa}

\def\H{\mathcal{H}}

\def\wt{\widetilde}
\def\mv{\mu_{{V}}}
\def\dmv{d \mu_{{V}}}
\def\Vsing{\|\d V\|_{\text{sing}}}
\def\B{\mathcal{B}}
\def\ov{\overline}

\def\res{\hbox{ {\vrule height .25cm}{\leaders\hrule\hskip.2cm}}\hskip5.0\mu}

\def\var{\underline{\underline{\text{v}}}}
\newtheorem{theorem}{Theorem}[section]
\newtheorem{lemma}[theorem]{Lemma}

\newtheorem{remark}[theorem]{Remark}
\newtheorem{corollary}[theorem]{Corollary}
\newtheorem{definition}[theorem]{Definition}

\newtheorem{claim}[theorem]{Claim}

\begin{document}

\begin{title}
{Allard-type boundary regularity for $C^{1,\a}$ boundaries}
\end{title}
\begin{author}
{Theodora Bourni}
\end{author}
\date{}
\maketitle
\vspace{-1cm}

\begin{abstract}
In this paper we show boundary monotonicity formulae for rectifiable varifolds having a $C^{1,\a}$ ``boundary". In particular, we show that the area ratios of balls centered at this ``boundary'' satisfy a nice monotonicity formula, similar to that for interior balls proved in \cite{al1}. This extends the boundary monotonicity formulae of Allard \cite{al2}, which require that the boundary is $C^{1,1}$.  As a corollary, the regularity results of \cite{al2} extend to this case and provide a regularity result for rectifiable varifolds with a $C^{1,\a}$ ``boundary''.
\end{abstract}

\section{Introduction}
In 1972 Allard \cite{al1} proved a remarkable regularity theorem for $k$-varifolds $V$ in $\R^{n+k}$. He showed that, under appropriate assumptions on the first variation and measure of $V$, $V$ is a $C^{1,\gamma}$ manifold, for some $\gamma\in (0,1)$. Replacing the varifold with a smooth manifold $M$ in $B_1(0)$, the unit ball in $\R^{n+k}$, his theorem  roughly says that if the mean curvature of $M$ is in  $L^p(\H^k)$, $p>k$, and if the area of $M$ is sufficiently close to that of a unit $k$-dimensional ball, then $M\cap B_{1/2}(0)$ is a graph of a $C^{1,\gamma}$ function with estimates, where $\gamma=1- k/p$. Later, in 1975, Allard \cite{al2} showed that this regularity result can be extended to $k$-varifolds with a $C^{1,1}$ ``boundary'', i.e. to  $k$-varifolds  that have bounded variation away from a $(k-1)$-dimensional $C^{1,1}$ manifold $B$, which we refer to as the boundary. Considering again the special ``smooth'' case; that is, when $M$ is a smooth $k$-dimensional submanifold of $B_1(0)\setminus B$ with $0\in B$, Allard's boundary regularity theorem roughly says the following. If  the mean curvature of $M$ is in  $L^p(\H^k)$, $p>k$, and if the area of $M$ is sufficiently close to that of a unit $k$-dimensional half-ball, then $M\cap B_{1/2}(0)$ is a graph of a $C^{1,\gamma}$ function with estimates, where $\gamma=1-k/p$.

These theorems are not only useful in regularity theory, but they also provide a powerful tool for compactness theorems for smooth manifolds. The purpose of this paper is to show that Allard's boundary regularity theorem still holds in the case of $C^{1, \a}$ boundaries for any $\a\in (0,1]$ (see Theorem ~\ref{main}).

One of the key ingredients of these regularity theorems is establishing {\it area monotonicity formulae}, which are tools that allow us to compare the measure, $\mv(B_r(x))$, of the varifold in ambient balls  with the area, $\o_k r^k$, of the corresponding Euclidean balls. In particular, these formulae  provide us with a   quantity that involves the ratios $\o_k^{-1}r^{-k}\mv(B_r(x))$ (called the ``area ratios'') and is monotone in $r$. To establish them, in the case of a varifold with ``boundary'' $B$,  i.e. when the varifold  has bounded variation away from $B$, one needs to show that the total variation of the varifold is actually a Radon measure (everywhere, not only away from $B$). 

The first part of this paper, presented in Section ~\ref{First variation and monotonicity}, is devoted to showing that the varifold has locally bounded first variation (everywhere) and that the area monotonicity formulae still hold in the case when the boundary $B$ is a $C^{1,\a}$ manifold for any $\a\in(0,1]$. The main difficulty in considering such boundaries is that, for $\a<1$, there is no neighborhood of $B$ on which the nearest point projection is well defined. Therefore, the function $\r_0(x)=\dist(x, B)$ is not necessarily differentiable almost  everywhere on the support of the varifold, a property which in Allard's paper \cite{al2} is extensively used to prove these monotonicity formulae. In our paper, we use a Whitney partition of $\R^{n+k}\setminus B$ to define a new ``distance'' function that is both smooth and also ``close'' enough to the standard distance $\r_0$ (see \eqref{rhot}). This allows us to carry out the necessary computations and estimates for the monotonicity formulae to hold. We also remark here that in \cite{al2}, even though all the estimates depend only on the $C^{1,1}$-norm of $B$,  it is always assumed that $B$ is smooth. In our paper, this new ``distance'' function allows us to drop this hypothesis; in particular, no proof requires higher than $C^{1,\a}$ regularity of the boundary.

Having established the monotonicity formulae, the proof of the boundary regularity theorem, which is stated in Section ~\ref{Boundary Regularity Theorem} (Theorem ~\ref{main}), follows the steps  of that of Allard's for $C^{1,1}$ boundaries, \cite[Section 4]{al2}. The proof is presented also in Section \ref{Boundary Regularity Theorem}. The proofs of many of the main ingredients, as one can see in Section \ref{Boundary Regularity Theorem}, are parallel to that of \cite{al2},
  with the exception of a height-tilt estimate, \cite[Lemma 4.5]{al2}, in whose proof the nearest point projection  is once again used. This lemma establishes a bound for the tilt-excess (in $L^2$) of the varifold, depending on the height-excess (in $L^2$), the mean curvature of the varifold and the $C^{1,1}$-norm of the boundary. Using again a Whitney partition  to replace the nearest point projection, we show that this theorem still holds in our case, with the bound now depending on the $C^{1,\a}$-norm of the boundary, rather than the $C^{1,1}$-norm (Theorem ~\ref{new tilt bound lemma}).

We begin, in Section ~\ref{notation and preliminaries}, by introducing some notation and establishing the setup with which we will be working.

\section{Notation and Preliminaries}\label{notation and preliminaries}
Let $n, k\in \N$, $n\ge 1$, $k\ge 2$ and let $B$ be a $C^{1,\a}$ closed $(k-1)$-dimensional submanifold of $\R^{n+k}$ passing through the origin, with $0<\a\le1$. 
Then, there exists a radius $R>0$ such that $B\cap B_{4R}(0)$ is a graph of a $C^{1,\a}$ function over $T_0B$, the tangent space of $B$ at $0$, and a non-negative constant $\k$ such that 
\begin{equation} \label{bdry smoothness}
\begin{split}
|\proj_{\N_bB}(y-b)|&\le\kappa |y-b|^{1+\a}\\
\|\proj_{\N_yB}-\proj_{\N_bB}\|&\le \k|y-b|^\a
\end{split}
\,\,\,\,\forall \, y,b \in B\cap B_{4R}(0). 
\end{equation}

We use the notation $T_xB$ for the tangent space of $B$ at $x$, $N_xB$ for the normal space of $B$ at $x$ and $\proj_{T_xB}$, $\proj_{N_xB}$ for the projections onto the two spaces respectively. Finally, $B_r(x)\subset\R^{n+k}$ will denote the $(n+k)$-dimensional ball of radius $r$ centered at $x\in\R^{n+k}$ and, for any $m\in\N$, $\o_m$ will denote the $m$-dimensional area of the open unit ball centered at the origin in $\R^m$.

\begin{definition}\label{distdef}For any $x\in B_{2R}(0)$ we define $\r_0(x)$ to be the distance of $x$ from $B$, i.e. $\r_0(x)=\dist(x, B)$, and $\bar{x}$ will denote a point on $B$ such that $|x-\bar{x}|=\r_0(x)$. \end{definition}
Note that there is not necessarily a unique such point $\bar{x}$, as in Definition ~\ref{distdef}. We also remark that for $x\in B_{2R}(0)$ any point $\bar x$, as in Definition ~\ref{distdef}, must be in $B_{4R}(0)$ and furthermore $x-\bar x\in N_{\bar x} B$.

Under the above assumptions it is easy to check that the following inequality holds.
\begin{remark}\label{prelimlem}
  Let $x\in B_{R}(0)\setminus B$ and $y\in B_{\r_0(x)/2}(x)$. Then 
  \begin{equation*} \label{bdry smoothness 2}
  |\proj_{N_{\bar{x}}B}(y-\bar{x})-(y-\bar{y})|\le c\k\r_0(y)^{1+\a}
\end{equation*}
 for some absolute constant $c$.
 To see this, note that
 \[\begin{split}|\proj_{N_{\bar{x}}B}(y-\bar{x})-(y-\bar{y})|&\le |(\proj_{N_{\bar{x}}B}-\proj_{N_{\bar{y}}B})(y-\bar{x})|+|\proj_{N_{\bar{y}}B}({y}-\bar {x}-(y-\bar y))|\\
&\le \k|\bar x-\bar y|^\a|y-\bar x|+\k|\bar y-\bar x|^{1+\a}\le 36\k\r_0(y)^{1+\a},
\end{split}\]
with the last inequality being a simple application of the triangle inequality.
\end{remark}

We consider a rectifiable $k$-varifold, $V=\var(M,\theta)$, where $M$ is a countably $k$-rectifiable, $\H^k$-measurable subset of $\R^{n+k}$ and $\theta$ a locally $\H^k$-integrable function on $M$, and we let
 $\mv=\H^k\res\theta$ be the weight measure of $V$ (cf. \cite[\S 15]{LSgmt}).  In this paper we always assume that the varifold $V$ satisfies
 \begin{equation}\label{P1}
0\in\spt V\,,\,\, \theta(x)\ge 1\text{  for  } \m_{V}\text{-almost every }x\in \R^{n+k}
\end{equation}
and  that the total variation measure of  $\d V$ (the first variation of $V$), when restricted to $B_R(0)\setminus B$, is a Radon measure. In particular, we assume that there exists a $\mv$-measurable function
$H:B_R(0)\setminus B\rightarrow\R^{n+k}$ with $|H(x)|=D_{\mv}\|\d V\|(x)$ (where $\|\d V\|$ is the total variation measure of $\d V$, cf. \cite[Section 4]{al1} or \cite[\S 39]{LSgmt}) for all $x\in
B_R(0)\setminus B$ such that
\begin{equation}\label{fvcpt}
\d V(X)=-\int_{B_R(0)} X\cdot H \dmv
\end{equation}
for any $C^1$ vector field $X$ with compact support in $B_R(0)$ and such that $X(y)=0$ for all $y\in B$. For $H$ we further assume that $H\in L^p(\mv\res(B_R(0)\setminus B))$ for some $p>k$ and we set
\begin{equation}\label{mc}
\left( \int_{B_R(0)\setminus B}|H|^p\dmv\right)^{\frac1p} =\Lambda.
\end{equation}

\begin{remark}\label{rescale}
Let $B$ and $V$ be as above (satisfying \eqref{bdry smoothness}-\eqref{mc}), with $\k=\k_1$, $R=R_1$, $\Lambda=\Lambda _1$ as defined in \eqref{bdry smoothness} and \eqref{mc}. A simple rescaling argument implies the following. Given $R_0>0$, $\k_0>0$ and $\Lambda _0> 0$  we can assume that $R_1= R_0$,  $\k_1\le \k_0$ and $\Lambda_1\le \Lambda _0$.

To see this, take $\theta=\max\left\{\frac{R_0}{R_1}, \left(\frac{\k_1}{\k_0}\right)^{\frac{1}{\a}}, \left(\frac{\Lambda_1}{\Lambda_0}\right)^{\frac{p}{p-k}}\right\}$. Rescaling $B$ and $V$ by $\theta$, we obtain $\wt B$, a $C^{1,\a}$ $(k-1)$-dimensional submanifold of $\R^{n+k}$ passing through the origin and $\wt V$ a rectifiable $k$-varifold respectively, satisfying \eqref{bdry smoothness}-\eqref{mc} in $B_{R_0}(0)\subset B_{\theta R_1}(0)$ and with
\[\k=\k_1\theta^{-a}\le \k_0\text{   and   }\Lambda=\Lambda_1\theta^{\frac kp-1}\le \Lambda_0.\]
This remark will allow us in what follows to assume, without loss of generality, that $R=1$ and furthermore that $\k, \Lambda$ are smaller that a chosen  constant.
\end{remark}

Throughout this paper the letter $c$ will denote a constant which possibly depends on the given variables $n, k, p, \a$. When different constants 
appear in the course of a proof we will keep the same letter $c$ unless the constant depends on some different parameters.

\section{First Variation and Monotonicity}\label{First variation and monotonicity}
Throughout this section we assume that $B$, $V$ are as defined in Section ~\ref{notation and preliminaries}, i.e. they satisfy \eqref{bdry smoothness} and properties \eqref{P1}, \eqref{fvcpt} and \eqref{mc}, for some $R, \k$ and $\Lambda$. We want to show that the total first variation of $V$ is a Radon measure in the whole ball $B_R(0)$. We will do this by using appropriate vector fields in the first variation that vanish on $B$. For this reason we need a smooth ``distance to $B$'' function. As mentioned in the introduction, the function $\rho_0(\cdot)=\dist(\cdot, B)$ is not differentiable everywhere on $\spt V$ and therefore we want to define a new smooth distance function that is ``close'' to $\r_0$.
We will need the following definitions.

 Let $\mathcal{W}$ be a Whitney partition of $B_{R}(0)\setminus B$ (cf. \cite[Chapter 5]{KrantzParks}). Then
  $$B_{R}(0)\setminus B\subset \cup_{\C\in \mathcal{W}}\C,$$
  where the elements $\C$ of the collection $\mathcal{W}$ are closed cubes satisfying  $\dist(\C,B)>0$ and
  $$\diam\C\le \dist(\C,B)\le 3\diam\C.$$
  For each $\C\in \mathcal{W}$, we let $x_\C\in \C$ be the center of the cube $\C$ and $y_\C\in B$ be a point such that $|x_\C-y_\C|=\r_0(x_\C)$ (using Definition ~\ref{distdef}, $y_\C=\bar x_\C$).
  Finally we let $\{\phi_\C\}_{\C\in \mathcal W}$ be a partition of unity
  suboordinate to the covering $\mathcal{W}$ and such that
  \begin{equation*}\label{DphiC}
  |D\phi_\C(x)|\le c\rho_0(x)^{-1}
  \end{equation*}
  where $c$ is an absolute constant (cf. \cite[Chapter 5]{KrantzParks}).
  
  We define now the function $\r_\C(x)=|\proj_{N_{y_\C}B}(x-y_\C)|$ for any $\C\in \mathcal W$,
  and define a new ``distance'' function by
  \begin{equation}\label{rhot}
  \r (x):=\left(\sum_{\C\in \mathcal{W}}\phi_C(x)\rho^2_\C(x)\right)^\frac12=\left(\sum_{\C\in \mathcal{W}}\phi_\C(x) |\proj_{N_{y_\C}B}(x-y_\C)|^2\right)^\frac12.
  \end{equation}
 Note that both $\rho_\C$ and $\rho$ are smooth in $B_{R}(0)\setminus B$. For $x\in \C$ we have
\[|x-x_\C|\le \frac{\diam \C}{2}\le \frac{\dist(\C,B)}{2}\le \frac{\r_0(x_\C)}{2}\]
and thus Remark ~\ref{bdry smoothness 2}, applied with $x$ and $y$ replaced by $x_\C$ and $x$ respectively, yields
\begin{equation}\label{rcr}\begin{split}
&|\r_\C(x)-\r_0(x)|\le c\k\r_0(x)^{1+\a}\text{      and  hence  }\\
 (1-c\k \r_0(x)^a)\r_0(x)&\le \r_\C(x)\le  (1+c\k \r_0(x)^a)\r_0(x),\,\,\forall x\in \C,
\end{split}\end{equation} 
where $c$ is an absolute constant. Assuming now that $\k$ is small enough, so that $c\k R^\a<\frac12$, with $c$ as in \eqref{rcr}, we obtain a relation between $\r_0$ and $\r$. In particular, we have that
\[\frac{\r_0(x)^2}{4}\le  (1-c\k \r_0(x)^\a)^2\r_0(x)^2\le\r^2(x)=\sum_{\C\in \mathcal{W}}\phi_C(x)\rho^2_\C(x)\le (1+c\k \r_0(x)^\a)^2\r_0(x)^2\le \frac{9\r_0(x)^2}{4}\]
and thus
\begin{equation}\label{rrhot}
 \frac{\r_0(x)}{2}\le(1-c\k \r_0(x)^\a)\r_0(x)\le \r(x)\le (1+c\k \r_0(x)^\a)\r_0(x)\le \frac{3\r_0(x)}{2}.
\end{equation}

  Furthermore, we have that
  
  \begin{equation}\label{rhoCDrhoC}
\r_\C(x)  D\r_\C(x)=\proj_{N_{y_\C}B}(x-y_\C)=: X_\C(x)
\end{equation}
and
\begin{equation*}
\begin{split}\rho(x) D\rho(x)=\frac12 D\rho^2(x)&= \sum_{\C\in \mathcal{W}}\phi_\C(x)\rho_\C(x)D\rho_\C(x)+\frac12\sum_{\C\in \mathcal{W}}D\phi_\C(x)\rho^2_\C(x)\\
&= \sum_{\C\in \mathcal{W}}\phi_\C(x) X_\C(x)+\frac12\sum_{\C\in \mathcal{W}}D\phi_\C(x)(\rho^2_\C(x)-\r_0^2(x)),
\end{split}
\end{equation*}
the last equality being true because $\sum_{\C\in \mathcal{W}}D\phi_\C(x)=0$.
Using \eqref{rcr}, we can estimate $|\rho^2_\C(x)-\r_0^2(x)|$ as follows.
\[\begin{split}|\rho^2_\C(x)-\r_0^2(x)|=&|\rho_\C(x)-\r_0(x)||\rho_\C(x)+\r_0(x)|\\
\le& c\k\r_0(x)^{1+\a} (2+c\k \r_0(x)^a)\r_0(x)\le c\k\r_0^{2+\a}(x),\,\,\forall x\in \C.
\end{split}\]
Hence, since $|D\phi_\C(x)|\le c\r_0(x)^{-1}$, we have that
\begin{equation}\label{rhoDrho}
\begin{split}\rho(x) &D\rho(x)= \sum_{\C\in \mathcal{W}}\phi_\C(x) X_\C(x)+Y(x),
\text{  where   }X_\C(x)=\proj_{N_{y_\C}B}(x-y_\C)\\&\text{  and           } \quad\quad|Y(x)|\le c\k  \r_0^{1+\a}(x)\le c\k  \r^{1+\a}(x),
\end{split}\end{equation}
with the last inequality being true by \eqref{rrhot}.
Finally, we note that \eqref{rhoDrho}, combined with \eqref{rcr} and \eqref{rrhot}, implies that
\[\begin{split}|\rho(x) D\rho(x)|\le (1+c\k \r_0(x)^\a)\r_0(x)+|Y(x)|&\le\frac{1+c\k \r_0(x)^\a}{1-c\k \r_0(x)^\a}\r(x)+c\k\r^{1+\a}(x)\\
&\le (1+4c\k \r_0(x)^\a)\r(x)+c\k\r^{1+\a}(x). \end{split}\]
Hence, using once more \eqref{rrhot}, we have
\begin{equation}\label{Drho}
|D\r(x)|\le 1+c\k \r(x)^\a.
\end{equation}

From now on  we assume that $\k$ is small enough so that $c\k R^\a<\frac12$, with $c$ as in \eqref{rcr}, and thus the above estimates \eqref{rrhot}, \eqref{rhoDrho} and \eqref{Drho} hold.
\begin{theorem} [First Variation Formula] \label{first variation formula}Let $B$ and $V$ be as defined in Section ~\ref{notation and preliminaries}, i.e. they satisfy \eqref{bdry smoothness} and properties \eqref{P1}, \eqref{fvcpt} and \eqref{mc}, for some $R, \k$ and $\Lambda$ and assume further that $c\k R^\a<1$, where $c$ is an absolute constant (see Remark ~\ref{fv rmk}). Then,
for any $C^1$ vector field $X$ with compact support in $B_R(0)$,
\begin{equation}\label{1st variation equation}
\d V(X)=-\int_{B_R(0)\setminus B} X\cdot H \dmv + \int_B X\cdot \eta \,d\Vsing,
\end{equation}
where $\eta$ is a $\|\d V\|$ measurable unit vector field, such that $\eta(y) \in N_yB$ for all $y \in B$.
\end{theorem} 
\begin{remark} \label{fv rmk}In Theorem ~\ref{first variation formula}, we specifically require that $\k, R$ are such that
$c\k R^\a<\frac12$, with $c$ as in \eqref{rcr} (so that the estimates \eqref{rrhot}, \eqref{rhoDrho} and \eqref{Drho} hold) and also $\k R^\a<\frac14$, so that  Claim ~\ref{ZZB}, which appears in the proof of the lemma, holds.
\end{remark}
\begin{proof}
We will first prove that $V$ has locally bounded variation in $B_R(0)$, i.e. we will show that for any compact subset $W\subset B_R(0)$ there exists a constant $c$ (depending on $W$) such that
\begin{equation}\label{bounded variation equation}
\d V(X)\le c \sup_{B_R(0)}|X|,
\end{equation}
for any $C^1$ vector field $X$ with support in $W$.

For any smooth function $\phi:\R\rightarrow \R$ we can write $\d V(X)$ as follows.
\begin{equation}\label{split 1st variation formula into 3 parts}
\begin{split}
\d V(X)=&\int_{B_R(0)} \dvg_M[(1-\phi(\r))X] \dmv+\int_{B_R(0)} \phi'(\r)\nabla\r\cdot X \dmv\\
&+\int_{B_R(0)}\phi(\r)\dvg_MX \dmv,
\end{split}
\end{equation}
where $\r=\r(x)$ is as defined in \eqref{rhot}.
Let $\{\phi_h\}_{0<h<1}$ be a family of smooth functions $\phi_h:\R\to\R$ such that 
\[\phi_h(t)=\begin{cases}1 \hbox{  for  }t\le h/2 \\ 0 \hbox{  for  }t\ge h\end{cases}, \,\phi_h'(t)\le0,\,|\phi'_h(t)|\le \frac3h,\] and so that
$\phi_h\stackrel{h\rightarrow 0}{\longrightarrow}\chi_{(-\infty,0]}$, the characteristic function of $(-\infty,0]$. 
Then, by property \eqref{fvcpt},
\begin{align*}
\int_{B_R(0)} \dvg_M[(1-\phi_h(\r))X] \dmv=
-\int_{B_R(0)} (1-\phi_h(\r))X\cdot H d\m_V \stackrel{h\rightarrow
0}{\longrightarrow}-\int_{{B_R(0)}\setminus B} X\cdot H d\m_V,
\end{align*}
since by \eqref{rrhot} we have that 
\[\lim_{h\to 0} \left\{x\in {B_R(0)}:\r(x)\ge \frac h2\right\}=\left\{x\in {B_R(0)}:\r_0(x)>0\right\}={B_R(0)}\setminus B.\]
Furthermore, since $\mv(B)=0$ we have
\begin{equation*}\
\int_{B_R(0)}\phi_h(\r)\dvg_MX \dmv\stackrel{h\rightarrow 0}{\longrightarrow}0,
\end{equation*}
where again we have used that by \eqref{rrhot} we have that
\[\lim_{h\to 0}\left\{x\in {B_R(0)}:\r(x)\le h\right\}= \{x\in {B_R(0)}:\r_0(x)\le 0\}=B\cap {B_R(0)}.\]
Hence, by using  \eqref{split 1st variation formula into 3 parts} with $\phi=\phi_h$ and letting $h\to 0$, we have that \eqref{bounded variation equation} is equivalent to showing that
\begin{equation}\label{bdry bounded variation of vector field}
\lim_{h\rightarrow 0}\frac1h\int_{T_h} |\nabla\r\cdot X| \dmv\le c\sup_{B_R(0)}|X|,
\end{equation}
where $T_h=\{x\in {B_R(0)}:\r(x)<h\}$. Here we have also used \eqref{mc}. Recalling the estimate \eqref{Drho}, we have $|\nabla\r|\le |D\r|\le 1+c\k \r^\a$.
Therefore for proving that $V$ has locally bounded variation in $B_R(0)$ it suffices to show that for any $W\subset\subset B_R(0)$ there exists a constant $c$ (depending on $W$) such that
\begin{equation}\label{bdry bounded variation of function}
\lim_{h\rightarrow 0}\frac1h\int_{T_h}\chi\dmv\le c,
\end{equation}
where $\chi$ is a  smooth function such that $\chi=1$ on $W$, $0\le\chi\le 1$ and with compact support in $B_R(0)$.

Given $\chi:B_R(0)\to\R$ a smooth non negative function with compact support, we define the vector field
  \begin{equation}
    \label{vector field X}
    X(x)=\sum_{\C\in\mathcal{W}}\phi_\C(x)\psi(\rho(x))\chi(x)X_\C(x),
  \end{equation}
  where   $\mathcal{W}$ is a Whitney partition of $B_{R}(0)\setminus B$, as defined in the beginning of this section, and
   $\psi:\R\to\R$ is a smooth
  non negative and non increasing function. Recall that $\phi_\C$ is a partition of unity
  suboordinate to  $\mathcal{W}$ and, for each $\C\in \mathcal{W}$, $X_\C$ is defined by
    \begin{equation}
    \label{vector field XC}
    X_\C(x)=\proj_{N_{y_\C}B}(x-y_\C),
  \end{equation}
where  $x_\C\in \C$ denotes the center of the cube $\C$ and $y_\C\in B$ is a point such that $|x_\C-y_\C|=\r_0(x_\C)$.
  $X$ is then a smooth vector field that vanishes on $B$ and thus by \eqref{fvcpt} 
  \begin{equation}
    \label{first variation of X}
\d V(X)=\int \dvg_MX\dmv=-\int X\cdot H\dmv.
  \end{equation}

Recalling the estimate \eqref{rhoDrho}, we have that
\begin{equation}\label{YC}
\sum_{\C\in\mathcal{W}}\phi_\C(x)X_\C(x)=\rho(x) D\rho(x)+ Y(x)\text{ , with $Y$ satisfying  }|Y(x)|\le c\k \rho(x)^{1+\a}.
\end{equation}
Furthermore,
$$\dvg_M X_\C=\trace(\pmb{m}(x)DX_\C)\ge 1$$
where $\pmb{m}(x)$ denotes the matrix of the projection onto $T_xM$ and $DX_\C$ is the matrix of the projection onto
$N_{y_\C}B$.
Therefore, (omitting the argument $x$) we obtain
\begin{equation*}
  \begin{split}
  \dvg_M X =\sum_{\C\in \mathcal{W}}\dvg_M\left(\phi_\C\psi(\rho)\chi X_\C\right)
  \ge& \chi \psi( \rho)+\chi\psi'(\rho)\nabla \rho\cdot \sum_{\C\in \mathcal{W}}\phi_\C X_\C
  \\
  &+\psi(\rho)\nabla\chi\cdot\sum_{\C\in \mathcal{W}}\phi_\C X_\C
 +\psi(\rho)\chi\sum_{\C\in \mathcal{W}}\nabla\phi_\C\cdot X_\C.\\
  \ge& \chi \psi( \rho)+\chi \rho|D\rho|^2\psi'(\rho) + \chi\psi'(\rho)(\nabla\rho\cdot Y- \rho|\nabla^\perp\rho|^2)\\
  &+\psi( \rho) \nabla\chi\cdot\left(\rho
 D \rho+Y\right)
+\psi(\rho)\chi\sum_{\C\in \mathcal{W}}\nabla\phi_\C\cdot X_\C.
\end{split}
\end{equation*}
By the estimate on $|D\r|$, \eqref{Drho},  we have
\begin{equation}\label{Z}
|D\r(x)|^2=1+\zeta(x) \text{ , with $\zeta$ satisfying  }|\zeta(x)|\le c\k \rho(x)^{\a},
\end{equation} 
and thus we obtain
\begin{equation}
  \label{summing terms in dvg}
  \begin{split}
  \dvg_MX\ge& \chi \psi( \rho)+\chi \rho \psi'(\rho) + \chi\psi'(\rho)(\rho \zeta+\nabla\rho\cdot Y- \rho|\nabla^\perp\rho|^2)\\
  &+\psi( \rho) \nabla\chi\cdot\left(\rho
 D \rho+Y\right)
+\psi(\rho)\chi\sum_{\C\in \mathcal{W}}\nabla\phi_\C\cdot X_\C.
\end{split}
\end{equation}

Note now, since $\sum_{\C\in\mathcal W}D\phi_\C(x)=0$ and $|D\phi_\C(x)|\le c\r(x)^{-1}$, and by \eqref{rcr} and \eqref{rrhot}, that
\begin{equation}
  \label{g}
  |\sum_{\C\in \mathcal{W}}\nabla\phi_\C\cdot X_\C|=\left|\sum_{\C\in\mathcal{W}}\nabla\phi_\C(x)\cdot\left(X_\C(x)-(x-\bar{x})\right)\right|\le c\k\rho^\a.
\end{equation}

Using the estimates \eqref{summing terms in dvg} and \eqref{g} in \eqref{first variation of X} we obtain
\begin{equation}
  \label{rho inequality}
  \begin{split}
  \int\chi(\psi(\rho)+\rho\psi'(\rho))\dmv
  \le&-\int\chi\psi(\rho)\rho D\rho\cdot H\dmv-\int\chi\psi(\rho)Y\cdot H\dmv\\
  &+\int\psi(\rho)\chi c\k\rho^\a\dmv-\int\psi(\rho)\nabla\chi(\rho D\rho+Y) \\
  &-\int\chi\psi'(\rho)(\rho \zeta+\nabla\rho\cdot
Y)\dmv+\int\chi\psi'(\rho)\rho|\nabla^\bot\rho|^2\dmv. 
\end{split}
\end{equation}

Let $\gamma:\R\rightarrow \R$ be a smooth function such that 
$$\gamma(t)=\begin{cases}1 \hbox{  for  }t\le
1/2\\ 0 \hbox{  for  }t\ge 1\end{cases}, \gamma'(t)\le 0 \,\,\, \forall t\,.$$ We will use
\eqref{rho inequality} with $\psi(\rho)=\gamma(\frac{\rho}{r})$ (for some $r>0$). Since
$$\rho\psi'(\rho)=\frac{\rho}{r}\gamma'\left(\frac{\rho}{r}\right)=-r\frac{\partial}{\partial r}
\left(\gamma\left(\frac{\rho}{r}\right)\right),$$ 
and because of the estimates  \eqref{YC}, \eqref{Drho} and \eqref{Z} for $|Y|$, $|\nabla \rho|$ and $|Z|$ respectively, this yields 
\begin{align}
  \label{gamma inequality}
\int\chi&\left(\left(1-c\k r^\a\right)\gamma\left(\frac{\rho}{r}\right)-\left(1+c\k r^\a\right)r\frac{\partial}{\partial r}
\left(\gamma\left(\frac{\rho}{r}\right)\right)\right)\dmv\notag\\
  \le&-\int\chi\gamma\left(\frac{\rho}{r}\right) \rho D\rho\cdot H\dmv -\int
  \chi \gamma\left(\frac{\rho}{r}\right) Y\cdot H\dmv\\
&-\int\gamma\left(\frac{\r}{r}\right) \nabla\chi\left(\rho D\rho+Y\right)\dmv
  -\int_U\chi
r\frac{\partial}{\partial r}
\left(\gamma\left(\frac{\rho}{r}\right)\right)|\nabla^\bot\rho|^2\dmv\notag.
\end{align}

Set 
\begin{equation}\label{Gamma}\Gamma=2c\k\a^{-1}, 
\end{equation}
where $c$ is the constant appearing on the LHS of \eqref{gamma inequality}. Note that 
\[\frac{e^{\Gamma r^\a}}{r}\frac{\partial}{\partial r}
\left(\gamma\left(\frac{\rho}{r}\right)\right)-e^{\Gamma r^\a}\frac{1-2c\k r^{\a}}{r^2}\gamma\left(\frac{\rho}{r}\right)=
\frac{\partial}{\partial r} \left(\frac {e^{\Gamma
r^\a}}{r}\gamma\left(\frac{\rho}{r}\right)\right).\] 
Hence, after multiplying \eqref{gamma inequality} by
$-(1+c\k r^\a)^{-1}e^{\Gamma r^\a}r^{-2}$ and noting that $-\frac{1-c\k r^\a}{1+c\k r^\a}\le-(1-2c\k r^\a)$, we obtain
\begin{align}
  \label{gamma inequality II}
  \frac{\partial}{\partial
  r}&\bigg(\frac{e^{\Gamma r^\a}}{r}\int\chi\gamma\left(\frac{\r}{r}\right)\dmv\bigg)\notag\\
  \ge&-\frac{e^{\Gamma r^\a}}{r^2}\int\gamma\left(\frac{\rho}{r}\right)\chi \rho |D \r||H|\dmv
  -\frac{e^{\Gamma r^\a}}{r^2}\int\gamma\left(\frac{\rho}{r}\right)\chi |Y||H|\dmv\\
  &  -\frac{e^{\Gamma r^\a}}{r^2}\int\gamma\left(\frac{\r}{r}\right)|\nabla\chi|\left(
  \rho|D\r|+|Y|\right)\dmv
  +\frac{1}{r}\int\chi
\frac{\partial}{\partial r}
\left(\gamma\left(\frac{\rho}{r}\right)\right)|\nabla^\bot\rho|^2\dmv\notag.
\end{align}

Letting $\gamma$ in \eqref{gamma inequality II}, increase to the characteristic function of $(-\infty,1)$ and
integrating  from $\sigma$ to $r$, where $0<\sigma<r$, we
obtain the following \textbf{monotonicity inequality for tubular neighborhoods of $B$}.
\begin{align}
  \label{monotonicity inequality}
  \frac{e^{\Gamma \s^\a}}{\sigma}\int_{T_\sigma}\chi\dmv\le&  \frac{e^{\Gamma r^\a}}
  {r}\int_{T_r}\chi\dmv\notag\\
  &+e^{\Gamma r^\a}\int_{T_r}\chi |D\r||H|\dmv+e^{\Gamma r^\a}\int_{T_r}\chi \frac{1}{\r}|Y||H|\dmv\\
&+e^{\Gamma r^\a}\int_{T_r}|D\chi|\left(|D\r|+\frac{1}{\r}|Y| \right)\dmv
  -\int_{T_r
  \setminus T_\sigma}\chi\frac{1}{\rho}|\nabla^\bot\rho|^2\dmv.\notag
\end{align}

Note that the last term on the RHS of \eqref{monotonicity inequality} is negative, and can therefore be dropped.
 The other terms are bounded because of the estimate for $|Y|$ and $|D\r|$ (see \eqref{YC}, \eqref{Drho}), the fact that $H$ is in $L^p$ (see \eqref{mc}) and the fact that $\chi$ can be chosen so that $|D\chi|$ is bounded, with the bound depending only on $W$ (recall that $\chi=1$ on $W$). Hence letting $\sigma\rightarrow 0$ yields
 \begin{equation}
  \label{bdry density}
  \lim_{h\rightarrow 0}\frac{1}{h}\int_{T_h}\chi\dmv\le c,
\end{equation}
 where the constant $c$ depends
on $ W$.
This proves \eqref{bounded variation equation}, i.e. that $V$ is of bounded variation or equivalently that $\|\d V\|$ is a Radon
measure. 
This implies that for any $C^1$ vector field $X$ with compact support in $B_R(0)$ we have that
$$\d V(X)=\int_{B_R(0)\setminus B}X\cdot H \dmv+\int_B X\cdot \eta \,d\Vsing,$$
where
\begin{equation}
\label{bdry variation as limit}
\int_B X\cdot \eta \,d\Vsing\le\lim_{h\rightarrow 0}\frac1h\int_{T_h}|X\cdot\nabla\r|\dmv
\end{equation}
(see \eqref{split 1st variation formula into 3 parts}, \eqref{bdry bounded variation of vector field}).

Finally we want to prove that $\eta=\eta(y)\in N_yB$, for all $y \in B$. We will do this by showing that the RHS of the inequality \eqref{bdry variation as limit} vanishes for all continuous  vector fields $X$ such that $X(x)\in T_x B$ for all $x\in B$. Note  that, by approximation, \eqref{bdry variation as limit} holds, not only for $C^1$, but also for  continuous vector fields $X$.

Let $Z_B$ be a continuous vector
field on $B$ with compact support in $B\cap B_R(0)$ and such that $Z_B(x)\in T_{{x}}B$ and $|Z_B(x)|\le 1$ for all $x\in B$. We will show that we can appropriately extend $Z_B$ in $B_R(0)$, so that when we apply it in 
\eqref{bdry variation as limit} the RHS vanishes.

\begin{claim}\label{ZZB} Assuming that $\k R^\a<\frac14$, there exists an extension $Z$ of $Z_B$ in $B_R(0)$ with the following properties. $Z$ is a continuous  vector field with compact support in $B_R(0)$, $|Z|\le 1$ and, for $h$ small enough (depending only on the support of $Z_B$),
\[\left|Z(x)\cdot \sum_{\C\in\mathcal W}\phi_\C(x) X_\C(x)\right|\le c\k\r(x)^{1+\a} \text{   in   } T_h\cap B_R(0),\]
where $c$ is an absolute constant and  $X_\C$ is as in \eqref{rhoCDrhoC}.
\end{claim}
We provide the proof of the claim at the end of the proof of the lemma.
Let $Z$ be as in Claim ~\ref{ZZB} and note also that, by approximation, \eqref{bdry variation as limit} holds, not only for $C^1$, but also for  continuous vector fields, and thus for this $Z$. We now estimate $|Z\cdot \nabla\r|$ as follows.
\begin{equation*}\begin{split}
|Z\cdot \nabla\r|^2= |Z\cdot (D\r-\nabla^\perp\r)|^2\le 2 |Z\cdot D\r|^2+2 |Z\cdot \nabla^\perp\r|^2
\end{split}
\end{equation*}
and 
\begin{equation*}\begin{split}
 |Z\cdot D\r|^2\le 
2\left|Z\cdot \left(D\r-\frac{1}{\rho}\sum_{\C\in\mathcal W}\phi_\C X_C\right)\right|^2+2\left|Z\cdot\frac{1}{\r}\sum_{\C\in\mathcal W}\phi_\C X_C\right|^2\le
c\k^2\rho^{2\a},
\end{split}
\end{equation*}
where we have used Claim ~\ref{ZZB} and the estimate \eqref{rhoDrho}. Hence
\begin{equation*}
|Z\cdot \nabla\r|^2\le
c\k^2\rho^{2\a}+2|\nabla^\bot\rho|^2
\end{equation*}
and so
\begin{equation*}
\frac1h\int_{T_h}|Z\cdot\nabla\r|\dmv \le \left(\frac{1}{h}\int_{T_h\cap\spt
Z}\dmv\right)^\frac 12 \left(2\int_{T_h\cap\spt
Z}\frac{|\nabla^\bot\r|^2}{\r}\dmv+\frac{c\k^2}{h}\int_{T_h\cap\spt Z}\r^{2\a}\dmv\right)^\frac12.
\end{equation*}
Using the monotonicity inequality \eqref{monotonicity inequality} with $\s\to0$ and the fact that $V$ is of bounded variation (in particular inequality \eqref{bdry density}) we have that for any $\chi$
with compact support
$$\int_{T_r}\chi\frac{|\nabla^\bot\r|^2}{\r}\dmv<\infty.$$
Hence,
\begin{equation*}
  \lim_{h\rightarrow 0}\frac{1}{h}\int_{T_h}|Z\cdot\nabla\r|\dmv=0.
\end{equation*}
Because of \eqref{bdry variation as limit}, this implies that $Z(y)\cdot\eta(y)=Z_B(y)\cdot\eta(y)=0$ for all $y\in B$ and so
$\eta(y)\in N_yB$ for all $y \in B$.

Finally, we give the proof of Claim ~\ref{ZZB}.

\noindent{\it Proof of Claim ~\ref{ZZB}.} By the assumption on $R$, $B\cap B_{4R}(0)$ is a graph of a $C^{1,\a}$ function over $T_0B$, the tangent space of $B$ at $0$. For any $x\in B_R(0)$, let $\omega(x)$ be the unique point in $B_{2R}(0)\cap B$ such that $x-\o(x)\in N_0B$ (cf. Definition \ref{ozc}). We define $Z$ by
\[Z(x)= Z_B(\omega(x))\zeta(\r(x)),\]
where $\zeta:\R\to[0,1]$ is a smooth function such that $\zeta(t)=1$ for $t\le h_0$ and $\zeta(t)=0$ for $t\ge 2h_0$, for some constant $h_0>0$. Then $Z$ is continuous,  satisfies $|Z|\le 1$ and has compact support in $B_R(0)$, provided that $h_0$ is small enough (depending on the support of $Z_B$). Furthermore, we have that
\[\left|Z\cdot\sum_{\C\in\mathcal W}\phi_\C X_\C\right|\le c\k \r^{1+\a}.\]
To show this last estimate, because of \eqref{rcr} and \eqref{rrhot}, it suffices to prove that
\[\left|Z(x)\cdot (x-\bar x)\right|\le c\k \r^{1+\a}.\]
Note that
\[\begin{split}\left|Z(x)\cdot (x-\bar x)\right|=&\left|\proj_{T_{\o(x)}B} Z(x)\cdot\proj_{N_{\bar x} B}(x-\bar x)\right|\\
&= \left|\proj_{T_{\o(x)}B} Z(x)\cdot\left(\proj_{N_{\bar x} B}(x-\bar x)-\proj_ {N_{\o(x)} B}(x-\bar x)\right)\right|\\
&\le\|\proj_{N_{\bar x} B}-\proj_ {N_{\o(x)} B}\|x-\bar x|\le c\k\r(x)|\bar x-\o(x)|^\a\\
&\le c\k\r^{1+\a}(x),
\end{split}\]
where we have used \eqref{bdry smoothness}, \eqref{rrhot} and the fact that 
\begin{equation}\label{c0}
|\bar x-\o(x)|\le |\bar x-x|+|x-\o(x)|\le (c_0+1)|\bar x-x|=(c_0+1)\r_0(x),
\end{equation}
for some constant $c_0=1+c\k R^\a$ (and thus it satisfies  $c_0\stackrel{\k\to 0}{\longrightarrow}1$).
 To see that \eqref{c0} holds, let 
\[\frac{|x-\o(x)|}{|\bar x-x|}=c_0\ge 1. \]
Using the notation $T=T_0B$, we then have
\[\begin{split}\frac{|\proj_{T^\perp}(\o(x)-\bar x)|}{|\proj_{T}(\o(x)-\bar x)|}\ge&\frac{|\proj_{T^\perp}(\o(x)-x)|-|\proj_{T^\perp}(x-\bar x)|}{|\proj_{T}(\o(x)- x)|+|\proj_{T}(x-\bar x)|}\\
\ge& \frac{|\o(x)-x|-|x-\bar x|}{|\proj_{T}(x-\bar x)|}\ge  \frac{(c_0-1)|x-\bar x|}{|x-\bar x|}=c_0-1,
\end{split}\]
and thus
\[|\proj_{T^\perp}(\o(x)-\bar x)|\ge (c_0-1)(|\o(x)-\bar x|-|\proj_{T^\perp}(\o(x)-\bar x)|).\]

Therefore, using again \eqref{bdry smoothness}, we obtain
\[\begin{split}
\frac{c_0-1}{c_0}|\o(x)-\bar x|&\le |\proj_{T^\perp}(\o(x)-\bar x)| \\
&\le|\proj_{N_{\o(x)}}(\o(x)-\bar x)|+|(\proj_{N_{\o(x)}}-\proj_{T^\perp})(\o(x)-\bar x)|\\
&\le \k|\o(x)-\bar x|^{1+\a}+\k|\o(x)|^\a|\o(x)-\bar x|\le 2\k R^\a|\o(x)-\bar x|
\end{split}\]
Since $\k R^\a<\frac14$, we have  $c_0\le \frac{1}{1-2\k R^\a}<1+4\k R^\a$, which shows the required estimate.
\end{proof}
\subsection*{Monotonicity Formulae}

We would like to derive now a monotonicity formula for the ratio
\[r^{-k}m(r):=\frac{1}{r^k}\mv\left(B_r(b)\right),\]
where $b$ is a given point on $B\cap B_R(0)$ and $r>0$ is such that $B_r(b)\subset B_R(0)$. Similar estimates can be found for example in \cite{al2, BROTH, DuzaarSteffen-bdryregularity, DuzaarSteffen-lambda, LSRHbdryregularity}.

Having established the first variation formula, Lemma
~\ref{first variation formula}, we can use it with the vector field
$$X(x)=\phi(d(x))(x-b),$$
where $d(x)=|x-b|$ and $\phi:\R\rightarrow\R$ is a smooth
function, so that $\phi\circ d$ has compact support in $B_{R-|b|}(b)$. We can then argue as in the interior case, i.e. when $B_r(b)\cap B=\emptyset$, see \cite[\S 17]{LSgmt}, keeping track now of the extra boundary term. Thus, by letting $\phi$ approach the characteristic function of
$(-\infty, r)$, where $0<r<R/4-|b|$, we obtain the following
\textbf{monotonicity identity}.
\begin{equation}\label{monotonicity identity}
\begin{split}
\frac{d}{dr}\left(r^{-k}\mv(B_r(b))\right)=&\frac{d}{dr}\int_{B_r(b)}\frac{|\nabla^\bot
d|^2}{d^k}\dmv +r^{-k-1}\int_{B_r(b)}(x-b)\cdot H\dmv
\\&+r^{-k-1}\int_{B\cap B_r(b)}(x-b)\cdot\eta(x) \,d\Vsing.
\end{split}
\end{equation}
Recall that $\eta(x)\in N_xB$ for all $x\in B\cap B_R(0)$ and so, using \eqref{bdry smoothness}, we can bound the last term on the RHS of \eqref{monotonicity identity} by 
\begin{equation*}
\left|\int_{B\cap B_r(b)}(x-b)\cdot\eta(x) d\Vsing\right|\le \k r^{1+\a} \Vsing(B_r(b)).
\end{equation*}

We now want to estimate $\Vsing(B_r(b))$. 
In the proof of the first variation formula we have shown that this singular measure is bounded by a limit of integration along tubular neighborhoods of $B$ (equation \eqref{bdry variation as limit} in the proof of Theorem~\ref{first variation formula}). Hence, using the monotonicity inequality for tubular neighborhoods (\eqref{monotonicity
inequality} in the proof of Theorem~\ref{first variation formula}) with $\chi$ approaching the characteristic function of $B_r(b)$ and
$\s\rightarrow 0$, and assuming that $\k R^\a$ is small enough (depending only on $\a$), so that the terms $e^{\Gamma r^\a}$, $|D\r|$ and  $|Y|$ appearing in the monotonicity inequality satisfy
$e^{\Gamma r^\a}\le 2$, $|D\r|+\r^{-1}|Y|\le 2$ (see \eqref{Gamma}, \eqref{Drho} and \eqref{YC}), we obtain the following estimate for the singular measure in terms of $m(r)$.
\begin{align}
  \label{bdry mass bound}
\Vsing(B_r(b))\le \frac{2}{r}m(r)+4\Lambda m(r)^\frac{1}{q}+4m'(r).
\end{align}
Using the notation
$$\bar{m}(r):=r^{-k}m(r),$$
we have that
\begin{equation}
  \label{bdry mass bound II}
 \Vsing(B_r(b))\le (4k+2)r^{k-1}\bar{m}(r)+4\Lambda r^\frac kq\bar{m}(r)^\frac{1}{q}+
  4r^k\bar{m}'(r).
\end{equation}

\begin{lemma}\label{p-monotonicity lemma}
Under the hypotheses of Theorem~\ref{first variation formula} and assuming that $c\k R^\a<1$, where $c$ is a constant that depends only on $\a$ (see Remark ~\ref{p-mono rmk}), the following holds. There exists a function $\Phi(r)$ and a constant $\Lambda_0=\Lambda_0(\Lambda, k,p,\a)$ such that for any $b\in B_R(0)\cap B$ 
\begin{equation*}\label{p-monotonicity formula}
e^{\Phi(r)}\bar{m}(r)^\frac1p +\Lambda_0 r^{1-\frac kp}
\end{equation*}
is an increasing function of $r$ for $r\in (0, R-|b|)$, where $\bar m(r)=r^{-k}m(r)=r^{-k}\mv(B_r(b))$.
In particular,
\[\Phi(r)=\frac{4k+2}{p\a}\k r^\a\,,\,\Lambda_0=\frac{\Lambda}{p-k}\exp\left(\frac{4k+2}{p\a}\right).\]
\end{lemma}

\begin{remark} \label{p-mono rmk} In Lemma ~\ref{p-monotonicity lemma}, we specifically require that $\k, R$ are such that Theorem~\ref{first variation formula} holds. We further assume that $\k R^\a<1$ and $\k R^\a$ is small enough (depending only on $\a$), so that the terms $e^{\Gamma r^\a}$, $|D\r|$ and  $|Y|$ appearing in the tubular monotonicity inequality (\eqref{monotonicity
inequality} in the proof of Theorem~\ref{first variation formula})  satisfy
$e^{\Gamma r^\a}\le 2$, $|D\r|+\r^{-1}|Y|\le 2$ (see \eqref{Gamma},  \eqref{Drho} and \eqref{YC}), so that
the estimates \eqref{bdry mass bound} and \eqref{bdry mass bound II} hold.
\end{remark}

\begin{proof}
By using the estimate for the singular measure \eqref{bdry mass bound II} in the monotonicity identity \eqref{monotonicity identity}, we obtain the following.
\begin{equation*}
\begin{split}
 \bar{m}'(r)\ge\frac{d}{dr}\int_{B_r(b)}&\frac{|\nabla^\bot
d|^2}{d^k}\dmv-(1+4\k r^\a )\Lambda r^{-\frac kp}\bar{m}(r)^\frac1q\\
  &-(4k+2)\k r^{\a-1}\bar{m}(r)-
  4\k r^\a\bar{m}'(r).
\end{split}
\end{equation*}
Since $\k r^\a<1$, we obtain
\begin{align*}
\bar{m}'(r)+\Lambda  r^{-\frac kp}\bar{m}(r)^\frac1q+(4k+2)\k r^{\a-1}\bar{m}(r)\ge\frac15\frac{d}{dr}\int_{B_r(b)}\frac{|\nabla^\bot
d|^2}{d^k}\dmv\ge 0.
\end{align*}
Multiplying by $\bar{m}(r)^{-\frac{1}{q}}$, we obtain
\begin{equation*}
\bar{m}'(r)\bar{m}(r)^{-\frac{1}{q}}+(4k+2)\k r^{\a-1}\bar{m}(r)^\frac{1}{p}\ge -\Lambda r^{-\frac {k}{p}}.
\end{equation*}
Let
\begin{equation*}\label{Phi of r}
\Phi(r)=\frac{4k+2}{p\a}\k r^\a.
\end{equation*}
By multiplying the above inequality by $p^{-1}e^{\Phi(r)}$, we then find 
\begin{equation*}
\left(e^{\Phi(r)}\bar{m}(r)^{\frac1p}\right)'\ge-\frac {\Lambda}{p} e^{\Phi(r)}r^{-\frac kp}\ge-\frac {\Lambda}{p}\exp\left(\frac{4k+2}{p\a}\right)r^{-\frac kp}.
\end{equation*}
Finally, letting
\begin{equation*}\label{Lambda}
\Lambda_0=\frac{\Lambda}{p-k}\exp\left(\frac{4k+2}{pa}\right),
\end{equation*}
we obtain
\begin{equation*}
\left(e^{\Phi(r)}\bar{m}(r)^{\frac1p}+\Lambda_0 r^{1-\frac kp}\right)'\ge0,
\end{equation*}
which proves the lemma.
\end{proof}
A consequence of Lemma 
~\ref{p-monotonicity lemma} is the following result about the density at points on $B$. 

\begin{corollary}[cf.  Section 3.5 in \cite{al2}]\label{density is upper semi continuous}
The density function
\begin{equation*}
\Theta_V(b)=\lim_{r\downarrow 0}\o_k^{-1}\bar{m}(r)
\end{equation*}
satisfies the following.
\begin{itemize}
\item[(1)] 
It is a real valued  upper semicontinuous function on $B\cap B_R(0)$.
\item[(2)] For all $b\in B\cap B_R(0)\cap \spt V$, $\Theta_V(b)\ge\frac 12$.
\item[(3)] There is a number $\mu>1$ with the property that if $b\in B\cap \spt V$ and $2\Theta_V(b)<\mu$, then there is a positive $r>0$ such that $B\cap B_r(b)\subset\spt V$.
\end{itemize}
\end{corollary}
\begin{proof}
The proof is identical to that of the theorem in \cite[Section 3.5]{al2}, using here Lemma ~\ref{monotonicity lemma} and \eqref{bdry smoothness}, 
instead of 3.4(2) and 2.2(4)(a) of \cite{al2}, respectively. 
\end{proof}

\begin{corollary}\label{lab}
Under the hypotheses of Lemma ~\ref{p-monotonicity lemma}, we have 
\[m(r)\ge\l r^k,\]
where $\l=\frac {1}{2^{p}} \exp\left(-\frac{4k+2}{\a}\right)\frac{\o_k}{2}$.
\end{corollary}
\begin{proof}
 By letting $\s\downarrow 0$ in the monotonicity formula of Lemma ~\ref{p-monotonicity lemma} and using the lower bound for the density at a boundary point (Corollary ~\ref{density is upper semi continuous}), we obtain  the following lower bound for $\bar{m}(r)$
\begin{equation*}
\bar{m}(r)^\frac 1p\ge\frac 12 \exp\left(-\frac{4k+2}{p\a}\right)\left(\frac{\o_k}{2}\right)^\frac 1p,
\end{equation*}
which implies the result.
\end{proof}

Finally, we want to prove one more monotonicity lemma.
\begin{lemma}\label{monotonicity lemma}
Under the hypotheses of Lemma ~\ref{p-monotonicity lemma} and assuming  that $c\k R^\a<1$ and $c\Lambda r^{1-\frac kp}<1$, where $c$ is a constant that depends only on $\a, p$ and $k$ (see Remark ~\ref{mono rmk}), the following holds. There exists a function $\Psi(r)$ such that for any $b\in B_R(0)\cap B\cap\spt V$ 
\begin{equation*}\label{first monotonicity equation}
e^{\Psi(\s)}\bar{m}(\s)\le e^{\Psi(r)}\bar{m}(r)- \frac12\int_{B_r(b)\setminus B_\s(b)}
  \frac{|\nabla^\bot d|^2}{d^k}\dmv
\end{equation*}
and
\begin{equation*}\label{second monotonicity equation}
  e^{-\Psi(\s)}\bar{m}(\s)\ge e^{-\Psi(r)}\bar{m}(r)-2\int_{B_r(b)\setminus B_\s(b)}
  \frac{|\nabla^\bot d|^2}{d^k}\dmv
\end{equation*}
 for all $0<\s<r<R-|b|$, where $\bar m(r)=r^{-k}m(r)=r^{-k}\mv(B_r(b))$.
In particular,
\[\Psi(r)=4(2k+1)\left(a\l^{-\frac 1p}\left(1-\frac kp\right)^{-1}r^{1-\frac kp}+\a^{-1}\k r^\a\right),\]
where $\l=\frac {1}{2^{p}} \exp\left(-\frac{4k+2}{\a}\right)\frac{\o_k}{2}$.
\end{lemma}

\begin{remark}\label{mono rmk} In Lemma \ref{monotonicity lemma}, we specifically require that $\k, R$ are such that both Theorem~\ref{first variation formula} and Lemma ~\ref{p-monotonicity lemma} hold. We further assume that $\k R^\a<\frac18$ and $\Lambda_0 R^{1-k/p}\le 1/2(\o_k/2)^{1/p}$, where $\Lambda_0$ is as in Lemma ~\ref{p-monotonicity lemma}, so that we are able to obtain a lower bound for $\bar m(r)$.
\end{remark}

\begin{proof}
By Corollary ~\ref{lab}, we have that
\[m(r)\ge\l r^k,\]
where $\l=\frac {1}{2^{p}} \exp\left(-\frac{4k+2}{\a}\right)\frac{\o_k}{2}$, and thus
\begin{equation}\label{H integral est}
\int_{B_r(b)}|H|\dmv\le \Lambda m(r)^{1-\frac 1p}\le \Lambda\l^{-\frac 1p}m(r) r^{-\frac kp}=\Lambda\l^{-\frac 1p}\bar{m}(r) r^{\frac kq}.
\end{equation}

We estimate the singular measure as we did in \eqref{bdry mass bound II}, using the monotonicity inequality for tubular neighborhoods (inequality \eqref{monotonicity inequality} in the proof of Theorem~\ref{first variation formula}), but now estimating the terms involving $H$ by \eqref{H integral est}, yielding
\begin{equation}\label{bdry mass bound III}
\Vsing(B_r(b))\le (4k+2)r^{k-1}\bar{m}(r)+4\Lambda\l^{-\frac 1p} r^\frac kq\bar{m}(r)+4r^k\bar{m}'(r).
\end{equation}
Using this estimate in the monotonicity identity \eqref{monotonicity identity}, we obtain
\begin{equation*}\label{m'bar upper bound}
\bar{m}'(r)-2(4k+2)\left(a\l^{-\frac 1p} r^{-\frac kp}+\k r^{\a-1}\right)\bar{m}(r)\le2\frac{d}{dr}\int_{B_r(b)}\frac{|\nabla^\bot
d|^2}{d^k}\dmv,
\end{equation*}
where we have used the assumption $4\k r^\a<1/2$. Similarly, we obtain
\begin{equation*}\label{m'bar lower bound}
\bar{m}'(r)+2(4k+2)\left(a\l^{-\frac 1p} r^{-\frac kp}+\k r^{\a-1}\right)\bar{m}(r)\ge\frac 12\frac{d}{dr}\int_{B_r(b)}\frac{|\nabla^\bot
d|^2}{d^k}\dmv.
\end{equation*}

Let
\begin{equation*}\label{Psi}
\Psi(r)=4(2k+1)\left(a\l^{-\frac 1p}\left(1-\frac kp\right)^{-1}r^{1-\frac kp}+\a^{-1}\k r^\a\right).
\end{equation*}
Multiplying the first inequality by $e^{-\Psi(r)}$ and the second  by $e^{\Psi(r)}$, we have that
\begin{equation*}
\left(e^{-\Psi(r)}\bar{m}(r)\right)'\le 2\frac{d}{dr}\int_{B_r(b)}\frac{|\nabla^\bot
d|^2}{d^k}\dmv,
\end{equation*}
and
\begin{equation*}
\left(e^{\Psi(r)}\bar{m}(r)\right)'\ge\frac 12\frac{d}{dr}\int_{B_r(b)}\frac{|\nabla^\bot
d|^2}{d^k}\dmv.
\end{equation*}
The lemma then follows by integrating these two inequalities.
\end{proof}

\section{Boundary Regularity Theorem}\label{Boundary Regularity Theorem}
In this section, we state and sketch the proof of the main regularity theorem, Theorem ~\ref{main}. 
We will use the following definition, borrowed from \cite{al2}, as it simplifies the exposition of the theorems and ensures our statements more closely resemble those of \cite{al2}.

\begin{definition}\label{B}
We say that $(V, B)\in \B(\d, \eta)$ if $B$, $V$ are as defined in Section ~\ref{notation and preliminaries} and 
\begin{itemize}
\item[(i)] they satisfy \eqref{bdry smoothness} and properties \eqref{P1}, \eqref{fvcpt} and \eqref{mc} with $R=1$ and $\k, \Lambda\le \eta$,
\item[(ii)] $\o_k^{-1}\mv(B_1(0))\le \frac{1+\d}{2}$, and
\item[(iii)] $T_0B=Y:=\R^{k-1}\times\{0\}^{n-k+1}$.
\end{itemize}
\end{definition}

\begin{remark}
For $\eta$  small enough, depending on $\a, p$ and $k$, the results of Section ~\ref{First variation and monotonicity} hold for any $(V, B)\in \B(\d,\eta)$. Remark ~\ref{rescale} explains why one can, without loss of generality, make the assumption that $R=1$ and $\eta$ is small. Finally, the assumption on $T_0B$ is not restrictive, as one can always achieve this by a rotation.
\end{remark}

The following definition will  also be needed later.
\begin{definition}\label{ozc}
For each $x\in\R^{n+k}\cap B_1(0)$, we define $\o(x)$ to be the unique point in $B\cap B_{\sqrt {1+\k^2}}(0)$
such that $x-\o(x)\in Y^\bot$ and $\zeta(x)$ to be the real valued function $\zeta(x)=|x-\o(x)|$. Furthermore, 
for all $x\notin B$ we define $\chi(x)$ to be the projection of $\R^{n+k}$ onto the subspace $Y+\{t(a-\o(a)):t\in \R\}$.
\end{definition}

We now state our main regularity theorem (see \cite[beginning of Section 4]{al2} for the case of a $C^{1,1}$ boundary).
\begin{theorem}\label{main} For any $\e\in (0,1)$, there exists $\d>0$, depending on $\a, p, k$ and $\e$, with the following property. If 
\begin{itemize}
\item[(1)] $0\le\eta\le \d$ and
\item[(2)]  $(V, B)\in \B(\d,\eta)$,
\end{itemize}
then, for some unit vector $u\in Y^\perp$, the following statements hold.
\begin{itemize}
\item[(i)] $\spt V$ has a unique tangent cone at $0$,  given by $\{y+tu:y\in Y, t\ge 0\}$,
\item[(ii)] $E_h:=\left(\int\dist(x, T)^2\dmv\right)^\frac12\le \e$, where $T=\{y+tu:y\in Y, t\in\R \}$,
\item[(iii)] $\spt V\cap (B_{1-\e}(0)\setminus B)$ is a continuously differentiable $k$-dimensional submanifold of $\R^{n+k}$ which is closed relatively to $B_{1-\e}(0)\setminus B$, whose closure contains $B_{1-\e}(0)\cap B$, and which  projects (under $\proj_T$) univalently on $\proj_T(\spt V\cap (B_{1-\e}(0)\setminus B))$,
\item[(iv)] $\|\proj_{T_x M}-\proj_T\|\le c\sup\{E_h,\eta\}$, $\forall x\in \spt V\cap (B_{1-\e}(0)\setminus B)$
and
\item[(v)]$\|\proj_{T_x M}-\proj_{T_yM}\|\le c\sup\{E_h,\eta\}|x-y|^\gamma$, $\forall x,y\in \spt V\cap (B_{1-\e}(0)\setminus B)$, where $\gamma=\min\{\a, 1-k/p\}$.
\end{itemize}
Here, in (iv) and (v), $c$ is a constant that depends only on $n,k, p, \a$ and $\e$ (and recall  {\emph{$V=\var(M,\theta)$}})
\end{theorem}

As mentioned in the introduction, using  the results  of Section ~\ref{First variation and monotonicity}, the proof of  Theorem ~\ref{main} follows  that of Allard's for $C^{1,1}$ boundaries, \cite[Section 4]{al2}. The proofs of many of the main ingredients are parallel to that of \cite{al2},  with the exception of a height-tilt estimate, \cite[Section 4.5]{al2} (see Definition ~\ref{excess}). However, for completeness  we present here all the lemmata needed for the proof, providing for each either a proof when required or referring to the corresponding one in \cite{al2} when the argument is identical.  
One first shows the following compactness result. For a sequence of  pairs $(V, B)\in \B(\d_i,\eta_i)$, such that $\d_i, \eta_i\to 0$, one can extract a subsequence that converges to a linear $k$-dimensional half space with boundary given by $Y$, \cite[Section 4.1]{al2}.  More precisely, we have the following.

\begin{lemma}[compactness] \label{compactness}The lemma in  \cite[Section 4.1]{al2} (i.e. 4.1(1) and 4.1(2) of \cite{al2}), with $\B(\d,\eta)$ as in Definition ~\ref{B}, holds.
\end{lemma}
\begin{proof} The proof is identical to that of   4.1(1) and 4.1(2) iof\cite{al2},  using here Corollary ~\ref{lab}
instead of 3.4(1) of \cite{al2}. 
\end{proof}

This compactness property is a very powerful tool, as it allows 
one to use arguments by contradiction; a technique which is used often throughout the proof. The first application of this compactness theorem, coupled with Allard's interior regularity theorem \cite{al1}, is to  prove an interior regularity lemma for a varifold with boundary. This lemma  provides not only $C^1$ regularity away from the boundary, but also a good geometric picture of that region; for instance, a good description of the tangent spaces (they are ``close'' to the spaces $\chi(x)$, defined in Definition ~\ref{ozc}) and ``smallness'' of the height-excess (see Definition ~\ref{excess}). More precisely,  we have the following.

\begin{lemma}[interior regularity] \label{interior} The lemma  in \cite[Section 4.3]{al2} (i.e. 4.3(1), 4.3(2), 4.3(3)  and 4.3(4) of \cite{al2}), with $\B(\d,\eta)$ as in Definition ~\ref{B}, holds.
\end{lemma}
\begin{proof} The proof is identical to that of 4.3(1)-(4) of \cite{al2},  using here  \eqref{bdry smoothness} instead of 2.1(1) and 2.2(4)(a) of \cite{al2}. We note also that 3.5(3) of \cite{al2} is Corollary ~\ref{density is upper semi continuous}(3) here.
\end{proof}

One then proceeds with what is known as a {\it Lipschitz approximation lemma}. This lemma states that ``most'' of $V$ can be approximated by the graph of a Lipschitz function, leaving out a part of ``large'' tilt-excess (see Definition \ref{excess}). It also provides ``good'' $L^2$ estimates of the approximating function around the boundary, something which is used later in the proof to further approximate this Lipschitz function by a harmonic one and as a consequence by a linear one. More precisely,  we have the following.

\begin{lemma}[Lipschitz approximation, cf.  Section 4.4 of \cite{al2}]\label{lipschitz approximation lemma}
There exists $\d>0$ such that if $0\le \eta\le \d$, $(V, B)\in \B(\d, \eta)$, $T$ is a $k$-dimensional linear space with $Y\subset T$ and  $N=\spt \mu_V\cap B_{15/16}(0)\setminus B$, then there are functions 
\[f: T\longrightarrow T^\perp\,\,\text{  and  }\,\, F: T\longrightarrow \R^n\]
with the following properties.
\begin{itemize}
\item[(1)] $F(z)=z+ f(z)$ for $z\in T$,
\item [(2)] $\lip f\le 1$,
\item [(3)] $\sup\{|f(z)|:z\in T\}\le \sup \{\dist (x, T):x\in N\}$ (recall  {\emph{$V=\var(M,\theta)$}}),
\item [(4)] for $\mu_V$ almost all $x\in N\setminus\image F$, $\|T_xM-T\|>1/22$ and
\item [(5)]
$\int_{\proj_T(N)\cap \{z:|\proj_{Y^\perp}(z)|<t\}}|f|^2d\H^k\le 2(t^2+2\eta t)\int_{\proj_T(N)}\|Df\|^2 d\H^k+4\eta^2t\o_{k-1}.$
\end{itemize}
\end{lemma}
\begin{remark} The above statement differs from that of \cite[Section 4.4]{al2} only in the constant appearing in the second term of the RHS of (5). This happens because of the different constants used in the definition of the ``regularity'' of the boundary---in particular, the $\k$ in \eqref{bdry smoothness} here versus the $\frac{\k}{2}$ in 2.1(1) of \cite{al2}---and a minor typographical error in the proof of \cite[Section 4.4]{al2} (page 435, line 7).
\end{remark}
\begin{proof}[Proof of Lemma ~\ref{lipschitz approximation lemma}]
The proof is identical to that of the lemma in \cite[Section 4.4]{al2}.  We remark that the results of \cite[Section 4.2]{al2} used in the proof are still valid in our case. In particular,  \cite[Section 4.2]{al2}, is used to prove that $|\o(x)-\o(a)|\le 2|x-a|$, $\forall\,x,a\in B_{1-\e}(0)$, with $\e\in (0,1)$ (see Definition ~\ref{ozc}). This is true in our case as well, because, by the triangle inequality and \eqref{bdry smoothness}, we have that
\begin{align*}
|\proj_{Y}(\o(x)-\o(a))|^2&\ge|\o(x)-\o(a)|^2-|\proj_{Y^\bot}(\o(x)-\o(a))|^2\ge\left(1-\k^2\right)|\o(x)-\o(a)|^2
\end{align*}
and hence $|\o(x)-\o(a)|\le(1-\k^2)^{-\frac12}|\proj_{Y}(x-a)|\le(1-\k^2)^{-\frac12}|x-a|$.
\end{proof}

The next main step is to prove a {\it height-excess decay lemma}. This lemma shows that the height-excess (see Definition ~\ref{excess}) has a ``nice'' decay as we pass to smaller balls; in particular, it decays as a power of the radius.  Before we proceed to this, we first need a further lemma that relates the tilt-excess (see Definition ~\ref{excess})  with the height-excess. The proof of the corresponding lemma in the case of a $C^{1,1}$ boundary,  \cite[Section 4.5]{al2}, as mentioned in the introduction, does not carry over in our case. The reason is that the proof consists of using the first variation formula with a vector field defined by the use of the nearest point projection. For a $C^{1,\a}$ boundary, we prove this result using a different method, introducing again a Whitney partition.  We now state and prove this lemma, providing first a necessary definition. 

\begin{definition}\label{excess}
  \label{tilt excess def} Let {\em $V=\var(M,\theta)$} be a rectifiable $k$-varifold in $\R^{n+k}$ and $T$  a $k$-dimensional subspace of $\R^{n+k}$.
 
  The \emph{tilt-excess} $E(x,r,T)$ of  $V$ with respect to $T$, in $B_r(x)$, is given by
  \[E(x,r,T)=r^{-k}\int_{B_r(x)}\|\proj_{T_{y}M}-\proj_{T}\|^2\dmv(y).\]
  
  The \emph{height-excess} $E_h(x,r,T)$ of  $V$ with respect to $T$, in $B_r(x)$, is given by
  \[E_h(x,r,T)=r^{-k-2}\int_{B_r(x)}\dist(x, T)^2\dmv(y).\]

\end{definition}

\begin{lemma}\label{new tilt bound lemma}
Assume that $B$, $V$ are as defined in Section ~\ref{notation and preliminaries}, i.e. they satisfy \eqref{bdry smoothness} and properties \eqref{P1}, \eqref{fvcpt} and \eqref{mc}, for some $R, \k$ and $\Lambda$.  Then, for any $r<R$ and any $k$-dimensional subspace $T$ with $T_0B\subset T$, we  have that
\begin{align*}\label{new tilt bound}
E(0,r/2,T)\le c\left(E_h(0,r,T)+
r^{2-k}\int_{B_r(b)}|H|^2\dmv+ (\k r^\a)^2 \right), 
\end{align*}
where $c$ is a constant that depends only on $n$ and $k$.
\end{lemma}
\begin{remark}\label{new tilt bound-p remark}
Applying H\"older's inequality to the RHS of the inequality of \emph{Lemma ~\ref{new tilt bound lemma}} yields
\begin{equation*}\label{new tilt bound-p}
\begin{split}
E(0,r/2,T)\le c \bigg(E_h(0,r,T)+
r^{2(1-\frac{k}{p})}\left(\int_{B_r(b)}|H|^p\dmv\right)^\frac 2p + (\k r^\a)^2\bigg).
\end{split}
\end{equation*}
Note that, in this computation, we have bounded the area ratios by a constant, which is be independent of $r$. For this, we require that $\k R^\a, \Lambda R^{1-\frac kp}$ are small enough (depending on $\a, p$ and $k$) so that the monotonicity formula of Lemma ~\ref{monotonicity lemma} holds.  
\end{remark}

\begin{proof}[proof of Lemma ~\ref{new tilt bound lemma}]

Without loss of generality we may assume that $T_0B=\R^{k-1}\times\{0\}^{n-k+1}$ and $T=\R^k\times\{0\}^{n-k}$, and we let
\[\bar B=B+(N_0B\cap T)=B+(\{0\}^{k-1}\times\R\times\{0\}^{n-k}).\]
Then, since $B$ satisfies \eqref{bdry smoothness} for some $R$ and $\k$, $\bar{B}$ is a $C^{1,\a}$, $k$-dimensional manifold for which the following hold. $\bar{B}\cap B_{4R}(0)$ can be written as the graph of a $C^{1,\a}$ function above $\R^k\times\{0\}^{n-k}=T_0\bar B$ and it satisfies \eqref{bdry smoothness}, with $B$ replaced by $\bar B$, and the same $R$ and $\k$. 
We also define  $ \bar \rho_0$ to be the distance from $\bar{B}$, i.e. ${\bar\rho_0}(x)=\dist(x,\bar{B})$. Note that  Remark ~\ref{prelimlem} still holds with $B$ replaced by $\bar{B}$, $\rho_0$ replaced by $\bar\rho_0$ and $\bar{x}$ denoting a point on $\bar B$ (instead of $B$ as usual, see Definition ~\ref{distdef}) such that $|x-\bar{x}|={\bar\rho_0}(x)$.

Let $\mathcal{W}$ be a Whitney partition of $\R^{n+k}\setminus \bar{B}$. Then
  \[B_R(0)\setminus \bar{B}\subset \cup_{\C\in \mathcal{W}}\C,\]
  where the elements $\C$ of the collection $\mathcal{W}$ are closed cubes satisfying  $\dist(\C,\bar{B})>0$ and
  $$\diam\C\le \dist(\C,\bar{B})\le 3\diam\C.$$
  Let $x_\C\in \C$ be the center of the cube $\C$ and $y_\C\in \bar{B}$ be such that $|x_\C-y_\C|=\bar\rho_0(x_\C)$.
  Finally let $\phi_\C$ be a partition of unity
  suboordinate to the covering $\mathcal{W}$ and such that
  \begin{equation*}
  |D\phi_\C(x)|\le c\bar\rho_0(x)^{-1},
  \end{equation*}
  where $c$ is an absolute constant.

  We define the following vector field.
  \[X=\zeta^2\sum_{\C\in\mathcal{W}}\phi_\C X_\C,\]
  where
  \[X_\C(x)=\proj_{T^\bot}(x-y_\C)\]
   and $\zeta$ is a smooth real valued function with compact support in $B_r(0)$ (where $r$ is as in the statement of the lemma) and such that
  \[\zeta(x)=1, \forall x\in B_{r/2}(0)\hbox{  and  } |D\zeta|\le 3/r.\]
  
  Since $X(x)=0$ for all $x\in B\cap B_R(0)$, the first variation formula \eqref{fvcpt} implies that
  \begin{equation}\label{first variation for X}
  \int \dvg_M X\dmv=-\int X\cdot H\dmv.
 \end{equation}
 We will estimate
 \begin{equation}\label{diviX}
 \dvg_M X= 2\zeta \sum_{\C\in\mathcal{W}}\phi_\C \nabla^M\zeta\cdot X_\C+ \zeta^2\sum_{\C\in\mathcal{W}}\nabla^M\phi_C\cdot X_\C+\zeta^2\sum_{\C\in\mathcal{W}}\phi_\C\dvg_M X_\C.
 \end{equation}
 For each $\C\in\mathcal{W}$ we have that
 \begin{equation*}
 \dvg_M X_\C=\frac 12|\proj_{T_xM}-\proj_{T}|^2.
 \end{equation*}
 To see this, let $\pmb{t}=(t^{ij})$ and $\pmb{m}=(m^{ij})$ denote the matrices of the projections onto $T$
  and $T_xM$ respectively. Then
  \begin{equation*}\label{nc}
\begin{split}    |\proj_{T_{x}M}-\proj_{T}|^2&=\sum_{j=1}^{n+k}e_j(\pmb{m}+\pmb{t}-2\pmb{m} \pmb{t})e_j
    =2k-2\sum_{j=1}^{n+k}e_j(\pmb{m}\pmb{t})e_j\\&=2\sum_{j=1}^{n+k}e_j
    (\pmb{m}(I-\pmb{t}))e_j
    =2\sum_{i=k+1}^{n+k}  m^{ii}.
  \end{split}
\end{equation*}
  To estimate the two first terms on the RHS of \eqref{diviX}, note that
  \begin{equation*}
  \begin{split}
\sum_{\C\in\mathcal{W}}\nabla^M\phi_\C\cdot X_\C&=\sum_{\C\in\mathcal{W}}\nabla^M\phi_\C \cdot\left(\proj_{T^\bot}(x-y_\C)-\proj_{T^\bot}(x-\bar{x})\right)\\
&=\sum_{\C \in\mathcal{W}}D\phi_\C\cdot(\proj_{T_xM}\circ\proj_{T^\bot})(\proj_{T^\bot}(\bar{x}-y_\C))
\end{split}
\end{equation*}
(where we recall that here $\bar{x}$ denotes a point on $\bar B$ such that $|x-\bar{x}|={\bar\rho_0}(x)$)
and
\begin{equation*}
\nabla^M\zeta\cdot X_\C(x)=D\zeta\cdot(\proj_{T_xM}\circ\proj_{T^\bot})(\proj_{T^\bot}({x}-y_\C)).
\end{equation*}
Hence, using the Cauchy-Schwartz inequality and Remark ~\ref{prelimlem}, we have that for any $\e>0$
\begin{equation*}\begin{split}
\left|\zeta^2\sum_{\C\in\mathcal{W}}\nabla^M\phi_C\cdot X_\C\right|&\le\e\zeta^2|\proj_{T_xM}-\proj_T|^2+ \frac{c}{\e}\zeta^2|D\phi_\C|^2|\proj_{T^\bot}({\bar x}-y_\C)|^2\\
&\le \e\zeta^2|\proj_{T_xM}-\proj_T|^2+ \frac{c}{\e}(\k r^{\a})^2,
\end{split}
\end{equation*}
since
\[\begin{split}
|\proj_{T^\bot}({\bar x}-y_\C)|^2&\le2 |(\proj_{T^\perp}-\proj_{N_{y_\C}\bar B})({\bar x}-y_\C)|^2+2 |\proj_{N_{y_\C}\bar B}({\bar x}-y_\C)|^2\\
&\le c\k^2 \bar\rho_\C^{2}r^{2\a}\le c\k^2 \bar\rho_0(x)^{2}r^{2\a}
\end{split}\]
and
\begin{equation*}
\begin{split}
|\zeta\nabla^M\zeta\cdot X_\C|&\le |\zeta||D\zeta||\proj_{T_xM}-\proj_T|(|\proj_{T^\bot}(x)|+|\proj_{T^\bot}(y_\C)|)\\
&\le \e \zeta^2|\proj_{T_xM}-\proj_T|^2+ \frac{c}{\e} r^{-2}\left(\dist(x,T)\right)^2+\frac{c}{\e}(\k r^\a)^2.
\end{split}
\end{equation*}

For estimating the RHS of the first variation formula \eqref{first variation for X} we note that
  \begin{equation*}
  \begin{split}
  |X_\C\cdot H|&\le r^{-2}|\proj_{T^\perp}(x-y_\C)|^2+r^2|H|^2\\
  &\le\frac{\dist(x,T)^2}{r^2}+(\k r^\a)^2+2r^2|H|^2.
  \end{split}
  \end{equation*}
    
  Hence, using the above estimates in the first variation formula \eqref{first variation for X}, for sufficiently small $\e>0$, we obtain the required estimate of the lemma 
    \begin{equation*}
  \begin{split}
  E(0,r/2,T)&\le c\left(r^{-k-2}\int_{B_r(b)}\dist(x,T)^2\dmv(x) +
r^{2-k}\int_{B_r(b)}|H|^2\dmv + (\k r^\a)^2\right).
  \end{split}
  \end{equation*}
\end{proof}

\begin{remark}\label{excessrmk} Lemma~\ref{new tilt bound lemma} provides a bound for $E(0, \theta r, T)$ with $\theta=\frac12$. We remark here that the same bound given in Lemma~\ref{new tilt bound lemma} is true for any $\theta\in (0,1)$. The proof is the same,  the only difference being that we should pick the function $\zeta$ such that
  \[\zeta(x)=1, \forall x\in B_{\theta r}(0)\hbox{  and  } |D\zeta|\le \frac{2}{(1-\theta)r}.\]
Of course, in this case the bound (in particular, the constant on the RHS of the estimate of Lemma \ref{new tilt bound lemma})  will depend also on $\theta$.

\end{remark}

We are now ready to state the height-excess decay lemma.
\begin{lemma}[height-excess decay lemma, cf. Section 4.7 of \cite{al2}]\label{tilt excess}
There exist $\theta,\Delta, C>0$ with $\theta\in (0,1)$, $C\in (1,\infty)$ and with the following property. If $(V, B)\in \B(\Delta, \eta)$, with $\eta\le \Delta$
and
\[\mu:=\left(\int_{B_1(0)}\dist^2(x, T)\dmv\right)^\frac12\le \Delta,\]
for some $k$-dimensional subspace $T\subset\R^{n+k}$ such that $Y\subset T$, then there exists $\wt T\subset \R^{n+k}$,
 a $k$-dimensional subspace  such that $Y
\subset \wt T$ and such that
\[\|\proj_T-\proj_{\wt T}\|\le C\mu\]
and
\begin{equation}\label{teeq}
\left(\theta^{-k-2}\int_{B_{\theta }(0)}\dist^2(x, \wt T)\dmv\right)^\frac12\le \max\{\theta^{1-\frac kp}, \theta^\a\}\max\{\mu, C\eta\}.
\end{equation}
\end{lemma}

\begin{remark} The only difference between the statement of Lemma ~\ref{tilt excess} and the corresponding one in Allard's paper  \cite[Section 4.7]{al2} is the rate of decay of the height-excess, which appears above in \eqref{teeq}. In particular, in \cite{al2}, since $\a=1$, the rate is $\theta^{1-\frac kp}$. However, here we have to account for the case when $\a<1-\frac kp$.
\end{remark}

\begin{proof}[Proof of Lemma ~\ref{tilt excess}] Having at our disposal Lemma ~\ref{new tilt bound lemma}, the proof of Lemma ~\ref{tilt excess} is the same as that of the lemma in \cite[Section 4.7]{al2}, with only two minor modifications concerning the choice of power of $\theta$, either $\a$ or $1-\frac kp$. These modifications occur at the beginning of the proof, where $\theta$ and $C_i$ are chosen (lines 2 and 6 of the proof of the lemma in \cite[Section  4.7]{al2}). In our case  one chooses $\theta$ so that $2^{k/2}D_4\theta <\max\{\theta^{1-\frac kp}, \theta^\a\}$ and $C_i$  so that
$\theta^{-(k+2)/2}\Delta_i^{-1}\le C_i\min\{\theta^{1-\frac kp}, \theta^\a\}$.
\end{proof}

With the use of the height-excess decay lemma, Lemma ~\ref{tilt excess}, the proof of the main regularity theorem, Theorem ~\ref{main}, follows almost as in the case of a $C^{1,1}$ boundary (see \cite[Sections 4.8, 4.9]{al2}). For completeness, and because the arguments in \cite[Sections 4.8, 4.9]{al2} are too compact but mostly because  in  \cite[Section 4.9]{al2} the nearest point (to the boundary) projection is once again used  (even though its regularity is not used), we provide here the proof.

\begin{proof}[Proof of Theorem ~\ref{main}]
Note first that under the assumptions (1) and (2) of Theorem ~\ref{main} and using Lemma ~\ref{interior} (in particular \cite[4.3(2)]{al2}), we can further assume the existence of a $k$-dimensional linear space $T$ that satisfies the hypotheses of the height-excess decay lemma, Lemma ~\ref{tilt excess}, i.e. $Y\subset T$ and
\[\mu=\left(\int\dist(x, T)^2\dmv\right)^\frac12\le \Delta,\] 
where $\Delta$ is as in Lemma ~\ref{tilt excess}. 
 We apply now the height-excess decay lemma, Lemma ~\ref{tilt excess}, iteratively, to conclude that there exists a $k$-dimensional linear space $\wt T$ such that $\|\wt T- T\|\le C\sup \{\mu, \eta\}$ and
\begin{equation}\label{Ttilde}\left(r^{-k}\int_{B_r(0)}\dist (x, \wt T)^2\dmv\right)^\frac12\le C\max\{\mu,\eta\}\max\{r^{2-\frac kp}, r^{1+\a}\}, \forall r\in (0,1)
\end{equation}
(cf. \cite[4.8(4) and proof of the lemma in Section 4.8]{al2}). This iteration is a standard technique that has been used also in the proof of Allard's interior regularity theorem \cite{al1} (for details see \cite[Section 8.17]{al1} or \cite[Proof of Theorem 23.1]{LSgmt}).  In the above inequality and from now on, we will use $C$ to denote any constant that depends on $n,k, p, \a$ and  $\e$, and we will not differentiate between the constants. 
Note that the above inequality, along with a simple rescaling argument, implies that any tangent cone of $V$ at zero is contained in $\wt T$ and, therefore, $V$ has a unique tangent cone at $0$ given by $\{y+tu:y\in Y, t\ge 0\}$, where $u\in \wt T\cap Y^\perp$. Hence we have shown that (i) and (ii) of Theorem \ref{main} hold with $T=\wt T$. 

We now want to show (iv) and (v) of Theorem \ref{main}. Note that statement (iii) of Theorem \ref{main} will then follow by (iv), Lemma~\ref{lipschitz approximation lemma} and Lemma ~\ref{interior}.
As in \cite[Section 4.9]{al2}, for any $b\in B$ we set $T_b= T_bB+ (N_bB\cap T)$. Then $T_b$ is a $k$-dimensional linear space such that $T_bB\subset T_b$ and
$\|\proj_{T_b}-\proj_{T}\|\le C\eta$ (because of  \eqref{bdry smoothness}),
which then implies that $\dist(x-b, T_b)\le \dist (x, T)+ C\eta$. Hence, for any $b\in B\cap B_{1-\frac\e2}(0)$ we have
\[\mu_b:=\left((1-|b|)^{-k-2}\int_{B_{1-|b|}(b)}\dist(x-b, T_b)^2\dmv\right)^\frac12\le C\e^{-\frac {k+2}{2}}(\mu +\eta).\]
 We can now apply the height-excess decay lemma, Lemma ~\ref{tilt excess}, in $B_{\frac{1-|b|}{2}}(b)$ (as we did in $B_1(0)$) provided that $\o_k^{-1}\left(\frac{1-|b|}{2}\right)^{-k}\mv\left(B_{\frac{1-|b|}{2}}(b)\right)\le \frac12+\Delta$, where $\Delta$ is as in Lemma ~\ref{tilt excess}.  For $\d$ small enough, a straightforward contradiction argument using Lemma \ref{compactness} implies that this is indeed  true for all $b\in B\cap B_{1-\frac\e2}(0)$ (this fact is also used in the proof of the lemma in \cite[Section 4.3]{al2}).
 Thus, arguing as for $\wt T$ above, we conclude that there exists a $k$-dimensional  linear space $\wt T_b$ such that $\|\proj_{\wt T_b}- \proj_{T_b}\|\le C\sup \{\mu, \eta\}$ and
\begin{equation}\label{exceqn}
\left(r^{-k}\int_{B_r(b)}\dist (x-b, \wt T_b)^2\dmv\right)^\frac12\le C\max\{\mu,\eta\}\max\{r^{2-\frac kp}, r^{1+\a}\},\forall r\in \left(0, \frac{1-|b|}{2}\right).
\end{equation}
We claim now that
\begin{equation}\label{wttwtt}
\|\proj_{\wt T_y}-\proj_{\wt T_b}\|\le C\max\{\mu, \eta\}\max\{|y-b|^{1-\frac kp}, |y-b|^\a\}\,,\,\,\forall y, b\in B\cap B_{1-\frac\e2}(0).
\end{equation}
Note first that it suffices to prove this for $y, b\in B$ such that $|y-b|<\frac\e8$, because else the claim is  true with an appropriately chosen constant $C$ (since  $\|\proj_{\wt T_b}-\proj_{\wt T_y}\|\le C\sup\{\mu, \eta\}$). If $|y-b|<\frac\e8$, then we can apply \eqref{exceqn} first with $B_r(b)$ replaced by $B_{2|y-b|}(b)$ and then with $B_r(y)$ replaced by $B_{2|y-b|}(y)$. The claim then follows directly by summing these two inequalities  and using a simple application of the triangle inequality (to be precise, we are summing the integrals on their common domain $B_{2|y-b|}(b)\cap B_{2|y-b|}(y)\supset B_{2|y-b|}(b)\cup B_{2|y-b|}(y)$).

We will show now that in any ball of the form $B_{\s(x)}(x)$, where $x\in B_{1-\e}(0)\setminus B$ and $\s(x)=\min\{\r_0(x), \frac\e8\}$, we can apply Allard's interior regularity theorem (see \cite[Section 8]{al1} or \cite[\S23]{LSgmt}). To do this, it suffices  to show that in such balls, the area ratios are close to 1 and that the tilt-excess over some $k$-dimensional linear space  is small.
More precisely, we will show the following. For any $\d_0$, we can pick $\d$ such that under the hypotheses of Theorem \ref{main}
\begin{equation}\label{claim1}
\o_k^{-1}\s^{-n}\mv(B_\s(x))\le 1+\d_0 \,,\forall x\in B_{1-\e}(0)\setminus B
\end{equation}
and
\begin{equation}\label{claim2}
E(x, \s, T_x)\le C\max\{\mu, \eta\} \max\{\s^{2(1-\frac kp)},\s^{2\a}\} \,,\forall x\in B_{1-\e}(0)\setminus B,
\end{equation}
where 
\[\s=\s(x)= \min\left\{\r_0(x), \frac\e8\right\}\text{ and  }T_x=\begin{cases}\wt T\,,\text{ if }\s(x)\ge \frac{\e}{16}
\\ \wt T_{\ov x}\,,\text{ if }\s(x)=\r_0(x)\end{cases}, (\ov x \text{ is as in Definition~\ref{distdef}}).\]
We first show \eqref{claim1},  arguing by contradiction and using Lemma~\ref{compactness} (similar arguments are used in the proof of the lemma in  \cite[Section 4.3]{al2}). Consider a sequence of pairs $(V_i, B_i)$ satisfying  the hypotheses of Theorem \ref{main} with corresponding $\d_i$'s converging to 0 and a sequence of points $x_i\in \spt V_i\cap B_{1-\e}(0)$ for which \eqref{claim1}  is not true. Then, after passing to a subsequence, $x_i\to x\in \ov B_{1-\e}(0)$. Applying now the compactness lemma, Lemma \ref{compactness}, to the sequence $(\eta_{\ov x_i, \s_i \sharp} V_i, \eta_{\ov x_i, \s_i }(B_i))$ if $\s_i=\s(x_i)\to 0$ (where $\ov x$ as in Definition \ref{distdef}), or to $(V_i, B_i)$ otherwise, we arrive at a contradiction. Here, for $x\in \R^{n+k}$ and $\l\in \R^+$, the function $\eta_{x,\l}:\R^{n+k}\to\R^{n+k}$ is given by $\eta_{x,\l}(y)=\l^{-1}(y-x)$, $\eta_{x,\l \sharp} V$ is the mapping of the varifold under $\eta_{x,\l}$ (see \cite[Section 4.2]{al1} or \cite[\S15]{LSgmt}) and $\ov x$ is as in Definition~\ref{distdef}.

To show \eqref{claim2}, we note first that if $\s(x)=\frac{\e}{16}$, then, by \eqref{Ttilde} and Lemma~\ref{new tilt bound lemma} (see also Remark~\ref{excessrmk}), we have
$E(x, \frac{\e}{16}, \wt T)\le c\e^{-k} E(0,1-\frac{\e}{16},\wt T)\le C\max\{\mu, \eta\}$. Assume now that $\s(x)=\r_0(x)$ and take $\ov x$ as in Definition \ref{distdef}. Using  Lemma~\ref{new tilt bound lemma} again and \eqref{exceqn}, we then have
$E(x,\r_0(x), \wt T_{\ov x})\le E(x, 2\r_0(x), \wt T_{\ov x})\le C\max\{\mu, \eta\}\max\{\r_0^{2(1-\frac kp)}, \r_0^{2\a}\}$. Hence the above claim is true.

Note now that \eqref{claim1} and \eqref{claim2}, along with Allard's interior regularity theorem (see \cite[Section 8]{al1} or \cite[\S23]{LSgmt}) applied to the balls $B_{\s(x)}(x)$ and a standard covering argument, implies that (iv) and (v) of Theorem~\ref{main} are true for $x, y\in \spt V\cap B_{1-\e}(0)$ such that $\r_0(x)\ge \frac{\e}{16}$, $\r_0(y)\ge \frac{\e}{16}$.

Now take any point $x\in\spt V\cap B_{1-\e}(0)$, such that $\r_0(x)\le \frac\e8$. By \eqref{claim1} and \eqref{claim2}, we can then apply Allard's interior regularity theorem (see \cite[Section 8]{al1} or \cite[\S23]{LSgmt}) to $B_{\r_0(x)}(x)$, which implies that for a constant $\gamma=\gamma(n,k,p)$, $\|\proj_{TyM}-\proj_{T_xM}\|\le C\max\{\mu,\eta\}|y-x|^{1-\frac kp}$, $\forall y\in B_{\gamma \r_0(x)}(x)$ and thus 
\begin{equation}\label{last}
E(x, \gamma\r_0(x), T_x M)\le C\max\{\mu, \eta\}\r_0^{1-\frac kp}.
\end{equation}
By  \eqref{claim2}, we also have
\[E(x, \gamma\r_0(x), \wt T_{\ov x} )\le\gamma^{-k} E(x, \r_0(x), \wt T_{\ov x} )\le C\max\{\mu, \eta\} \max\{\r_0(x)^{2(1-\frac kp)},\r_0(x)^{2\a}\}.\]
Summing these two inequalities, we obtain 
\begin{equation}\label{last2}
\|\proj_{T_xM}-\proj_{\wt T_{\ov x}}\|\le C\max\{\mu, \eta\} \max\{\r_0(x)^{1-\frac kp},\r_0(x)^{\a}\}.
\end{equation}
 Note that  here we have used the interior monotonicity formula (see \cite[Theorem 5.1(1)]{al1} or \cite[Theorem 17.6]{LSgmt}) to obtain a lower bound for the area. Estimate \eqref{last2}, along with \eqref{wttwtt}, imply (iv) of Theorem~\ref{main} for all $x\in \spt V\cap B_{1-\e}(0)$ such that $\r_0(x)\le\frac\e8$. This concludes the proof of (iv)  of Theorem~\ref{main}.

Finally, to show (v) of Theorem~\ref{main}, we take $x,y\in B_{1-\e}(0)$ such that $\r_0(x), \r_0(y)\le \frac\e8$. If $|y-x|\le \gamma\r_0(x)$, where $\gamma=\gamma(n,k,p)$ is as in \eqref{last}, then (v) of Theorem~\ref{main} is clear by \eqref{last}. If $|y-x|>\gamma\r_0(x)$, then,  using \eqref{last2} (applied to both $x$ and $y$) and \eqref{wttwtt}, we estimate as follows.
\begin{equation*}\label{last3}
\begin{split}
\|\proj_{T_xM}-\proj_{T_yM}\|\le& \|\proj_{T_xM}-\proj_{\wt T_{\o (x)}}\|+\|\proj_{\wt T_{\o (x)}}-\proj_{\wt T_{\o( y)}}\|\\
&+\|\proj_{T_yM}-\proj_{\wt T_{\o (y)}}\|\\
\le& C\max\{\mu, \eta\}\left( \max\{\r_0(x)^{1-\frac kp},\r_0(x)^{\a}\}+\max\{\r_0(y)^{1-\frac kp},\r_0(y)^{\a}\}\right)\\
&+C\max\{\mu, \eta\}\max\{|\o (x)-\o (y)|^{1-\frac kp}, |\o (x)-\o (y)|^\a\},
\end{split}
\end{equation*}
since
\[\begin{split} \|\proj_{T_xM}-\proj_{\wt T_{\o (x)}}\|&\le \|\proj_{T_xM}-\proj_{\wt T_{\ov x}}\|+\|\proj_{\wt T_{\ov x}M}-\proj_{\wt T_{\o (x)}}\|\\
&\le C\max\{\mu, \eta\} \max\{\r_0(x)^{1-\frac kp},\r_0(x)^{\a}\},
\end{split}\]
where we have used the estimate $|\ov x-\o(x)|\le c\r(x)$ (see \eqref{c0}).
We estimate   $\r_0(y)$ and $|\o(x)-\o(y)|$ in terms of $|x-y|$ as follows.  
\[\begin{split}\r_0(y)&\le  |y-\o(y)|\le |x-y|+|\o (x)-\o (y)|+|x-\o(x)|\\
&\le |x-y|+|\o (x)-\o (y)|+c\r(x)\le c|x-y|+|\o (x)-\o (y)|\le c|x-y|,
\end{split}\]
where at the last step we used the estimate $|\o(x)-\o(y)|\le c|x-y|$ (see proof of Lemma ~\ref{lipschitz approximation lemma}). Putting everything together, we obtain
\begin{equation*}
\begin{split}
\|\proj_{T_xM}-\proj_{T_yM}\|\le C\max\{\mu, \eta\}\max\{| x- y|^{1-\frac kp}, | x- y|^\a\}.
\end{split}
\end{equation*}
This concludes the proof of (v) of Theorem~\ref{main} and  completes the proof  of Theorem~\ref{main}.
\end{proof}

\section*{Acknowledgments}
I would like to thank Professor Leon Simon for suggesting this problem to me and for all his helpful advice. Furthermore, I would like to thank Ulrich Menne and Alexander Volkmann for  fruitful discussions on Allard's papers \cite{al1, al2} and I would further like to thank Ulrich Menne  for all his useful suggestions in writing this paper. Finally, I would like to thank Mat Langford for his time in proofreading this paper.

\bibliographystyle{plain}
\bibliography{bibliography}

\end{document}